\documentclass[a4paper,twoside,11pt,reqno]{amsart}
\newcommand{\Ueberschrift}{Convex Fujita numbers are not determined by the fundamental group}
\newcommand{\Kurztitel}{Convex Fujita numbers and $\pi_1$}
\usepackage{amsmath} \usepackage{amssymb} \usepackage{amsopn}
\usepackage{amsthm} \usepackage{amsfonts} \usepackage{mathrsfs}
\usepackage{latexsym}
\usepackage[headings]{fullpage}
\usepackage{mathtools}
\usepackage[all]{xy}
\usepackage[dvips]{graphicx} 
\usepackage[usenames]{color}
\usepackage[T1]{fontenc}
\usepackage[utf8]{inputenc}
\usepackage{mathabx}
\usepackage{stmaryrd}

\usepackage{setspace,fancyhdr}
\usepackage[margin=2.5cm]{geometry}
\usepackage[nobysame,alphabetic]{amsrefs}
\usepackage{framed}
\usepackage{tikz}
\usepackage{enumitem}
\usepackage{booktabs}
\usepackage[pdfpagelabels, pdftex]{hyperref}
\hypersetup{
  pdftitle={},
  pdfauthor={},
  pdfsubject={},
  pdfkeywords={},
  colorlinks=true,    
  linkcolor=blue,     
  citecolor=blue,     
  filecolor=blue,      
  urlcolor=blue,       
  breaklinks=true,
  bookmarksopen=true,
  bookmarksnumbered=true,
  pdfpagemode=UseOutlines,
  plainpages=false}
\DeclareMathOperator{\rH}{H}

\DeclareMathOperator{\rh}{h}

\newcommand{\bC}{{\mathbb C}}

\newcommand{\bP}{{\mathbb P}}

\newcommand{\bZ}{{\mathbb Z}}

\newcommand{\cL}{{\mathscr L}}
\newcommand{\cM}{{\mathscr M}}

\newcommand{\cO}{{\mathscr O}}

\newcommand{\dM}{{\mathcal M}}

\newcommand{\dO}{{\mathcal O}}

\newcommand{\inj}{\hookrightarrow}

\DeclareMathOperator{\pr}{pr}

\DeclareMathOperator{\Hom}{Hom}

\DeclareMathOperator{\Aut}{Aut}

\DeclareMathOperator{\Spec}{Spec}

\DeclareMathOperator{\Div}{Div}

\DeclareMathOperator{\Pic}{Pic}

\newcommand{\redu}{{\rm red}}

\DeclareMathOperator{\Alb}{Alb}
\newcommand{\Gm}{\bG_m}

\DeclareMathOperator{\NS}{NS}

\DeclareMathOperator{\R}{R}

\DeclareMathOperator{\Hilb}{Hilb}

\DeclareMathOperator{\Sym}{Sym}

\newtheorem{thm}{Theorem}[section]

\newtheorem{prop}[thm]{Proposition}
\newtheorem{lem}[thm]{Lemma}

\newtheorem{cor}[thm]{Corollary}
\newtheorem{conj}[thm]{Conjecture}

\newtheorem{thmABC}{Theorem}

\newtheorem{corABC}[thmABC]{Corollary}
\newtheorem{propABC}[thmABC]{Proposition}

\theoremstyle{definition}
\newtheorem{defi}[thm]{Definition}

\theoremstyle{remark}
\newtheorem{rmk}[thm]{Remark}
\newtheorem{rem}[thm]{Remark}

\newtheorem{ex}[thm]{Example}

\newenvironment{pro*}[1][Proof]{{\it{#1:}} }{}
\newenvironment{pro**}[1][]{{\it{#1}} }{\hfill $\square$}

\numberwithin{equation}{section}
\newcommand{\tref}[1]{Theorem~\ref{#1}}
\newcommand{\secref}[1]{\S\ref{#1}}

\newcommand{\cref}[1]{Corollary~\ref{#1}}
\newcommand{\conjref}[1]{Conjecture~\ref{#1}}

\newcommand{\lref}[1]{Lemma~\ref{#1}}
\newcommand{\pref}[1]{Proposition~\ref{#1}}

\newcommand{\eps}{\varepsilon}

\def\mc{\mathscr}

\def\Gm{\mathbb{G}_m}

\def\pr{\text{pr}}

\def\P{\mathbb{P}}
\def\bP{\P}
\def\Q{\mathbb{Q}}
\def\R{\mathbb{R}}

\def\cL{\mc{L}}
\def\cM{\mc{M}}

\def\cO{\mc{O}}

\DeclareMathOperator{\conFN}{Fu} 

\newcommand{\dleq}{\ensuremath{\,\leq\,}}
\newcommand{\deq}{\ensuremath{\stackrel{\textrm{def}}{=}}}

\definecolor{intOrange}{rgb}{1.0,.310,.0}
\usepackage[textsize=tiny]{todonotes}

\begin{document}

\hrule width\hsize

\vskip 0.5cm

\title[\Kurztitel]{\Ueberschrift} 

\author{Jiaming Chen}
\address{Jiaming Chen, Institut f\"ur Mathematik, Goethe--Universit\"at Frankfurt, Robert-Mayer-Stra\ss e {6--8},
60325~Frankfurt am Main, Germany} 
\email{\tt chen@math.uni-frankfurt.de}

\author{Alex K\"{u}ronya}
\address{Alex K\"uronya, Institut f\"ur Mathematik, Goethe--Universit\"at Frankfurt, Robert-Mayer-Stra\ss e {6--8},
60325~Frankfurt am Main, Germany} 
\email{\tt kuronya@math.uni-frankfurt.de}

\author{Yusuf Mustopa}
\address{Yusuf Mustopa, University of Massachusetts Boston, Department of Mathematics, Wheatley Hall, 100 William T Morrissey Blvd, Boston, MA 02125, USA}
\email{Yusuf.Mustopa@umb.edu}

\author{Jakob Stix}
\address{Jakob Stix, Institut f\"ur Mathematik, Goethe--Universit\"at Frankfurt, Robert-Mayer-Stra\ss e {6--8},
60325~Frankfurt am Main, Germany} 
\email{\tt stix@math.uni-frankfurt.de} 
	
\thanks{The authors acknowledge support by Deutsche Forschungsgemeinschaft  (DFG) through the Collaborative Research Centre TRR 326 "Geometry and Arithmetic of Uniformized Structures", project number 444845124.}
	
\maketitle

\date{\today} 

\maketitle

\begin{quotation} 
\noindent \small {\bf Abstract} --- We study effective global generation of adjoint line bundles on smooth projective varieties. To measure the effectivity we introduce the concept of the convex Fujita number of a smooth projective variety and compute its value for a class of varieties with prescribed dimension $d \geq 2$ and an arbitrary projective group as fundamental group.
\end{quotation}

\DeclareRobustCommand{\SkipTocEntry}[5]{}
\setcounter{tocdepth}{1} {\scriptsize \tableofcontents}

\section{Introduction}
\label{sec:intro}

\subsection{History and Motivation}
Our aim is to study the relation between the topology of algebraic varieties and effective positivity results  for line bundles on them. The model statement we consider is Lefschetz' theorem for ample line bundles on abelian varieties: the tensor square of an ample line bundle is globally generated, the tensor cube is very ample, independently of the dimension of the underlying manifold.   
In our current work, we will look at the question how the fundamental group of an algebraic variety influences the positivity of adjoint line bundles. 

Constructing global sections of line bundles or having effective control over such is an ancient and consistently difficult problem in algebraic geometry. In the last decades, much of the effort directed towards this problem was guided by the conjectures of Fujita (for global generation and very ampleness  \cite{fujita_polarized_1987}) and Mukai (for the study of higher syzygies, see \cite[Conjecture ~4.2]{ein_syzygies_1993}). 

The main purpose of our work is to study effective global generation of convex integral linear combinations of ample divisors while being able to determine precise Fujita-type bounds in a number of cases. In particular, we succeed in constructing examples with the following prescribed parameters: dimension, fundamental group, and global generation behavior. We essentially  exhaust the range of these invariants simultaneously given the global generation behavior predicted by Fujita's freeness conjecture.

Fujita's conjectures on global generation anticipate that given a smooth projective variety $X$ equipped with an ample Cartier divisor $L$, the adjoint divisor $K_X+mL$ should be base point free provided $m\geq \dim(X)+1$.  This  is a classical consequence of Riemann--Roch if $\dim(X)=1$, and has been demonstrated for $\dim(X)\leq 5$ \cites{reider_vector_1988,ein_global_1993,kawamata_fujitas_1997,helmke_fujitas_1997,ye_fujitas_2020}. For arbitrary dimension, there exist non-trivial global generation results due to Angehrn--Siu \cite{angehrn_effective_1995}, Heier \cite{heier_effective_2002},  and Ghidelli and Lacini \cite{ghidelli_logarithmic_2021}. It is important to remember that these bounds are uniform but nevertheless not linear in $\dim(X)$.

Effective global generation behavior of divisors on a given variety has not been explicitly known except in a handful of cases. Fujita's freeness conjecture is sharp for hyperplane divisors on projective spaces $\bP^n$; on the other end of the spectrum, a classical result of Lefschetz shows that given an ample divisor  $L$ on an abelian variety $X$, the divisor $2L$ is globally generated independently of the dimension of $X$. 
Related results have been obtained by Pareschi--Popa \cite[Theorem 5.1]{pareschi_regularity_2003} on the global generation of $2K_X + 2L$ for a nef and big divisor $L$ on an irregular variety $X$ with finite Albanese morphism.  In the context of varieties $S$ with numerically trivial canonical divisors, there are results for the Hilbert scheme $\Hilb^2(S)$ of subschemes of length $2$, by Riess~\cite{ries_non-divisorial_2021} when $S$ is a K3 surface of degree $2$, 
and by K\"uronya--Mustopa \cite{kuronya_effective_2022} when $S$ is an abelian surface.

In order to be able to talk about intermediate bounds on effective global generation and treat global generation in higher codimensions, the authors of \cite{kuronya_effective_2022} introduced Fujita numbers  for coherent sheaves with a view towards the codimension of the base locus. 
It is this line of thought that we follow, except that we focus solely on the Fujita number 
describing global generation itself --- being the closest in spirit to the freeness conjecture --- while introducing the (topological) fundamental group of the variety  as an extra parameter, and concentrating on the finer variant using convex linear combinations instead of multiples of a given ample divisor. 

\subsection{Convex Fujita numbers}
As mentioned above, we will measure effective positivity of  line bundles via the global generation of the associated sequence of adjoint ones, as  suggested by Fujita's conjectures.  Throughout this paper we will work with varieties over $\bC$.

\begin{defi}
The \textbf{convex Fujita number} of a smooth projective variety $X$ is the minimal $m \geq 0$ such that  for all $s \geq m$ and any ample divisors $L_1, \ldots, L_s$ on $X$ the adjoint divisor 
\[
K_X + L_1 + \ldots + L_s
\]
is globally generated. We will denote the convex Fujita number by $\conFN(X)$, or,
if no such $m$ exists, then we set $\conFN(X) = \infty$ (but this does not occur, see \pref{prop:FNisFinite}).
\end{defi}

Fujita conjectured in \cite{fujita_polarized_1987} that for a smooth projective variety $X$ of dimension $n$ and any ample divisor $L$ on $X$ the multiple adjoint divisor $m( K_X + t L)$ is globally generated if $m + t > n+1$ and $K_X+ tL$ is nef. Fujita moreover shows in the same paper that $K_X + tL$ is nef for $t \geq n+1$. As a result, only the case $m=1$ has been traditionally referred to as  follows.

\begin{conj}[\textbf{Fujita's freeness conjecture} \cite{fujita_polarized_1987}]
\label{conj:FujitaFreeness}
Let $X$ be a smooth projective variety and let $L$ be an ample divisor on $X$. Then for all $m \geq \dim(X) + 1$ the adjoint divisor $K_X + m L$ is globally generated.
\end{conj}

Fujita's freeness conjecture follows directly from the Riemann--Roch theorem for curves; for surfaces it is a quick consequence of Reider's theorem obtained by vector bundle techniques  \cite{reider_vector_1988},  Ein and Lazarsfeld in \cite{ein_global_1993} showed it for threefolds, Kawamata in \cite{kawamata_fujitas_1997} for $4$-folds and Ye and Zhu in \cite{ye_fujitas_2020} for 5-folds. All results for $\dim(X) \geq 3$ rely on vanishing theorems and non-klt center methods. The conjecture is currently open in dimensions six and above.  

Uniform bounds for $m$ have been proven such that $K_X + m L$ is globally generated for all ample divisors $L$ on all smooth projective varieties $X$ of a given dimension $n$. Angehrn and Siu show in \cite[Theorem 0.1]{angehrn_effective_1995}  that $m \geq (n^2 + n + 2)/2$ suffices, while Heier \cite[Theorem 3.1]{heier_effective_2002} improved the bound to $m \geq (e + 1/2)n^{4/3} + 1/2 n^{2/3} + 1$, where $e = \exp(1)$ is Euler's number. More recently, Ghidelli and Lacini proved in  \cite[Theorem 1.1]{ghidelli_logarithmic_2021} an asymptotically much better bound of 
\[
m \geq \max \{n + 1, n(\log \log(n) + 2.34)\}.
\]
We will discuss  in \secref{sec:FujitaFinite} the impact on the convex Fujita number of some of the methods and results  on Fujita's freeness conjecture as recalled above.

A numerical version of \conjref{conj:FujitaFreeness} was proposed by Helmke \cite[Conjecture 1.2]{helmke_fujitas_1997} as follows.
\begin{conj}[\textbf{numerical Fujita's freeness conjecture} \cite{helmke_fujitas_1997}]
\label{conj:numericalFujitaFreeness}
Let $X$ be a smooth projective variety of dimension $\dim(X) = n$ and let $L$ be an ample divisor on $X$ such that 
\begin{enumerate}[label=(\roman*),align=left,labelindent=0pt,leftmargin=*,widest = (iii)]
\item
$(L^n) > n^n$, and 
\item
for all irreducible cycles $Z \subseteq X$ of dimension $\dim(Z) = d<n$ we have $(L^d \cdot  Z) \geq n^d $.
\end{enumerate}
Then the adjoint divisor $K_X+ L$ is globally generated.
\end{conj}

The numerical Fujita's freeness conjecture implies $\conFN(X) \leq \dim(X) + 1$ for all $X$ (see
\pref{prop:FNisFinite} for the elementary argument showing a weaker but unconditional bound), and this bound in turn obviously implies Fujita's freeness conjecture. We may wonder whether 
\[
\conFN(X) \leq \dim(X) + 1
\]
always holds.  As we will see, it holds for all the examples constructed in this paper whose convex Fujita number is calculated precisely.

\subsection{Varieties with prescribed convex Fujita numbers}

\begin{defi}[Fujita simple and Fujita extreme varieties]
We say that a smooth projective  variety $X$ is 
\begin{itemize}
	\item \emph{Fujita simple} if  $\conFN(X) = 0$, and 
	\item \emph{Fujita extreme} provided $\conFN(X) \geq \dim(X) + 1$.  
\end{itemize}	

\end{defi}

\begin{ex}
For $X = \bP^n$ all ample line bundles are isomorphic to $\dO(a)$ for some $a> 0$.  It follows at once that $\conFN(\bP^n) = n+1$, so Fujita extreme varieties exist in all dimensions. 
\end{ex}

\begin{ex}
Riemann-Roch implies that all smooth projective curves $C$ have $\conFN(C) = 2$ regardless of the genus of $C$. The line bundles $\cL$ on $C$ such that $\omega_C \otimes \cL$ is not globally generated are precisely the line bundles $\cL = \dO_C(P)$ for an arbitrary point $P \in C$.
\end{ex}

\begin{rem}
The canonical bundle of a Fujita simple variety is by definition globally generated and thus nef. In particular, Fujita simple varieties are minimal.
\end{rem}

\begin{ex}
Its a classical theorem of Lefschetz that on an  abelian variety $A$ already $\cL^{\otimes 2}$ is globally generated for any ample line bundle $\cL$. The proof of Lefschetz's theorem generalizes to estimate the convex Fujita number of abelian varieties by $\conFN(A) \leq 2$.  This was proved in \cite[Theorem 1.1]{bauer_tensor_1996}. It is 
based on an application of the theorem of the square 
\[
\cL^{\otimes 2} \simeq t_{x}^\ast \cL \otimes t_{-x}^\ast \cL,
\]
Kodaira vanishing and Riemann-Roch $\rh^0(A,\cL) = \chi(A,\cL) = \frac{1}{g!} (\cL^g) > 0$
(applied to the translates $t_{x}^\ast \cL$ and $t_{-x}^\ast \cL$ instead of $\cL$), and the product map
\[
\bigoplus_{x \in A} \rH^0(A,t_{x}^\ast \cL ) \otimes \rH^0(A,t_{-x}^\ast \cL) \to  \rH^0(A,\cL^{\otimes 2}).
\]
\end{ex}

We recall that a group $\pi$ is called \textbf{projective} if it is isomorphic to the (topological) fundamental group $\pi_1(X)$ of a smooth projective variety $X$.

\begin{thmABC}[see \pref{prop:FNsimpleSurface} and \pref{prop:FNsimpleVariety}]
\label{thmABC:FNsimple}
Let $\pi$ be a projective group, and let $n \geq 2$ be an integer. There is a Fujita simple smooth connected projective variety $X$ of dimension $n$ with 
\begin{enumerate}[label=(\roman*),align=left,labelindent=0pt,leftmargin=*,widest = (iii)]
\item 
$X$ is of general type, and 
\item
$\pi_1(X)$ is isomorphic to $\pi$. 
\end{enumerate}
\end{thmABC}
The varieties $X$ that prove \tref{thmABC:FNsimple} are constructed in two ways. The first method starts with the surface case $n=2$ (dealt with in \pref{prop:FNsimpleSurface}) and then takes the product of a Fujita simple surface and a simply connected Fujita simple variety, more precisely a hypersurface in  $\bP^{n-1}$ of degree at least $n+1$.  This covers dimensions $n \geq 4$. 

The second method covers $n \geq 3$ and constructs $X$ as  a cyclic cover $f: X \to Y$ that totally ramifies along a smooth ample divisor, such that the degree $d = \deg(f)$ is sufficiently large. The variety $Y$ is a complete intersection of the correct dimension $n$ in $S \times \bP^n$ with a smooth projective variety $S$ that has the correct fundamental group $\pi_1(S) = \pi$. Such an $S$ exists because $\pi$ is assumed to be a projective group.

More precisely, analyzing complete intersections in $\bP^N$ leads to examples for $n \geq 3$ and for $(n, k) = (2,0)$ of the following theorem. The remaining examples in dimension $n=2$ are obtained as the blow-up of $\bP^2$ in $3-k$ points.

\begin{thmABC}[see \tref{thm:FNsimplyconnected}]
\label{thmABC:FNsimplyconnected}
Let  $n \geq 2$ be an integer, and let $0 \leq k \leq n+1$. There is a simply connected smooth connected projective variety $X$ of dimension $n$ with convex Fujita number $\conFN(X) = k$. Moreover, if $k=0$, then $X$ can be chosen to be of general type.
\end{thmABC}

The behavior of convex Fujita numbers in products is mysterious in general, but easy in the following special case. It will be applied in the case of a simply connected variety $Y$.

\begin{propABC}[see \cref{cor:FNproductwithPn}]
\label{propABC:FNproductwithPn}
Let $X$ be a smooth projective variety, and let $Y$ be a smooth projective variety with $\rH^{1}(Y,\dO_{Y}) = 0$.  Then 
\[
\conFN(X \times Y) = \max\{\conFN(X), \conFN(Y)\}.
\]
\end{propABC}

The K\"unneth formula for $\pi_1$, \pref{propABC:FNproductwithPn} and the case $k=0$ of \tref{thmABC:FNsimplyconnected} reduce \tref{thmABC:FNsimple} to the case of dimension $n= 2$ and $3$. Even more follows by combining of \tref{thmABC:FNsimple} with the full version of \tref{thmABC:FNsimplyconnected}.

 Our next result indicates that the topological invariant $\pi_1(X)$ alone is not sufficient to control positivity properties of adjoint line bundles.

\begin{thmABC}[see \tref{thm:anyFNanyPi1}]
\label{thmABC:anyFNanyPi1}
Let $\pi$ be a projective group, and let $n \geq 2$ be an integer, and let $0 \leq k \leq n-1$. There is a smooth connected projective variety $X$ of dimension $n$ with the following properties:
\begin{enumerate}[label=(\roman*),align=left,labelindent=0pt,leftmargin=*,widest = (iii)]
\item
$\pi_1(X)$ is isomorphic to $\pi$, and 
\item
$\conFN(X) = k$.
\end{enumerate}
\end{thmABC}

The surface with prescribed fundamental group and convex Fujita number $1$ is obtained as the blow up in one point of a carefully chosen surface with the same fundamental group and Fujita number $0$. The threefold with convex Fujita number $\conFN(X) = 2$ arises as $X = S \times \bP^1$ with a surface $S$ with the same fundamental group and convex Fujita number $\conFN(S) \leq 2$. The threefold $X'$ with convex Fujita number $\conFN(X') = 1$ arises as a branched double of the threefold $X$ with  $\conFN(X) \leq 2$.

\smallskip

\tref{thmABC:anyFNanyPi1} has the following obvious corollary.

\begin{corABC}
	The precise value of the Fujita number $\conFN(X)$ is not determined by the fundamental group $\pi_1(X)$ alone.
\end{corABC}

\subsection{Acknowledgements}
We thank the anonymous referee for a careful reading of the paper, and for suggesting a stronger version of our original \pref{prop:FNisFinite}.

The authors acknowledge support by Deutsche  Forschungsgemeinschaft  (DFG) through the Collaborative Research Centre TRR 326 "Geometry and Arithmetic of Uniformized Structures", project number 444845124.  Part of this work was done while the third author attended the workshop "Birational Complexity of Algebraic Varieties" at the Simons Center for Geometry and Physics, and he would like to thank the organizers and staff for the hospitality and stimulating atmosphere.

\section{Preliminaries on convex Fujita numbers}
\label{sec:FNprelims}

\subsection{Finiteness} 
\label{sec:FujitaFinite}

We first show that convex Fujita numbers are finite. We start with surfaces, $3$-folds and toric varieties in order to review that the typical techniques in the Fujita freeness conjecture specific to these dimensions also bound the convex Fujita number.

\begin{prop}
\label{prop:FNsurfacesReider}
Let $X$ be a smooth projective surface. 
\begin{enumerate}[align=left,labelindent=0pt,leftmargin=*]
\item 
\label{propitem:reider1}
The convex Fujita number of $X$ is bounded by $\conFN(X) \leq 3$.
\item 
\label{propitem:reider2}
If $\conFN(X) = 3$, then there exist an ample divisor $L$ on $X$ with $(L^2) = 1$.
\item 
\label{propitem:reider3}
If the intersection pairing on the N\'eron-Severi lattice $\NS(X)$ is even, then $\conFN(X) \leq 2$.
\item 
\label{propitem:reider4}
If the canonical divisor is numerically equal to $2 D$ 
with $D \in \Div(X)$, then $\conFN(X) \leq 2$.
\end{enumerate}
\end{prop}
\begin{proof}
By the Riemann-Roch formula, the assumption in \eqref{propitem:reider4} implies the assumption in \eqref{propitem:reider3} and clearly thus both then follow from \eqref{propitem:reider1} and \eqref{propitem:reider2}. 

Let us prove \eqref{propitem:reider1}. We have to show a bound on $m$ such that for ample divisor $L_1, \ldots, L_m$ the adjoint divisor $K_X + L$ with $L = L_1 + \ldots + L_m$ is globally generated. If $m \geq 3$, then $(L^2) \geq 9$ and by \cite[Theorem 1 (i)]{reider_vector_1988} base points can only occur if there is an effective divisor $C$ such that $(C^2) \leq 0$ and $(L \cdot C) - (C^2) = 1$. But $(L \cdot C) \geq m$ since the $L_i$ are ample, contradiction. 

For \eqref{propitem:reider2} we must have a base point of $K_X + L$ when $m=2$. Reider's method as in the proof of \eqref{propitem:reider1} still works unless $(L^2) \leq 4$, which implies $(L_i \cdot L_j) = 1$ for all $i, j$. This shows \eqref{propitem:reider2}.
\end{proof}

The following consequence of Reider's method will be used in our proof of \tref{thmABC:FNsimple}.

\begin{prop}
\label{prop:FNsurfacesReider divisible intersection numbers}
Let $X$ be a smooth projective surface. 
If the intersection pairing on the N\'eron-Severi lattice $\NS(X)$  takes values in $d\bZ$ for some $d \geq 5$, then the convex Fujita number of $X$ is bounded by $\conFN(X) \leq 1$.
\end{prop}
\begin{proof}
We must show that for any ample divisor $L$ on $X$ the adjoint divisor $K_X + L$ is globally generated. Since by assumption $(L^2) \geq d \geq 5$, we can apply \cite[Theorem 1 (i)]{reider_vector_1988} and must exclude the exceptional case: there is an effective divisor $C$ such that $(C^2) \leq 0$ and $(L \cdot  C) - (C^2) = 1$. But by assumption $d$ divides $(L \cdot  C) - (C^2)$, so we are done.
\end{proof}

\begin{prop}
Let $X$ be a smooth projective threefold. Then the convex Fujita number is bounded by $\conFN(X) \leq 4$.
\end{prop}
\begin{proof}
Let $\cL$ be a sum $L_1 + \ldots + L_m$ of $m \geq 4$ ample divisors $L_i$ on $X$. Then for all $i, j, k$ and all surfaces $S$ and all curves $C$ in $X$ we have
\[
(L_i \cdot L_j \cdot L_k) \geq 1, \quad (L_i \cdot L_j \cdot S) \geq 1, \quad (L_i \cdot C) \geq 1,
\]
so that 
\[
(L^3) \geq m^3 > 27, \quad (L^2 \cdot S) \geq m^2\geq 9, \quad (L \cdot C) \geq m \geq 3.
\]
The estimate $\conFN(X) \leq 4$ now follows from \cite[Theorem 5.2]{helmke_fujitas_1997}.
\end{proof}

Fujita's freeness conjecture has been proved for toric varieties in  \cite[Theorem 0.3]{laterveer_linear_1996}. The bound for the convex Fujita number of toric varieties has also been established.

\begin{prop}
For a toric variety $X$ the convex Fujita number is bounded by 
\[
\conFN(X) \leq \dim(X) + 1.
\]
Equality holds only for $X = \bP^n$.
\end{prop}
\begin{proof}
In view of $\omega_X = \dO_X(-\sum_i D_i)$ with $D_i$ being the torus invariant prime divisors, the result follows at once from \cite[Theorem 0.3]{mustata_vanishing_2002}\footnote{Musta\c{t}\u{a} describes this theorem as a strong version of Fujita's freeness conjecture. The notion of convex Fujita number provides a conceptual framework for this kind of strengthening of Fujita's conjecture. The result can also be extracted from \cite{laterveer_linear_1996}, but the chosen reference is more direct.}. 
\end{proof}

Results leading to general bounds on Fujita numbers tend to deliver explicit statements on convex Fujita numbers as well; in particular, they imply that $\conFN(X)$ is always finite. The work of Angehrn--Siu yields the bound 
$\conFN(X) \leq  \frac{1}{2}(n^2 + n + 2)$ which is quadratic in the dimension $n = \dim(X)$. Here we establish the appropriate version of the theorem of Ghidelli--Lacini \cite{ghidelli_logarithmic_2021}.

\begin{prop}
\label{prop:FNisFinite}
For a smooth projective variety of dimension $n$ the convex Fujita number is bounded as
\[
\conFN(X) \leq 1+ n( \log(\log n) + 2.34)\ .
\]
More precisely, we even have $\conFN(X) \leq n+1$ for $n \leq 4$.
\end{prop}
\begin{proof}
Let $L_1,\dots,L_m$ be ample integral divisors on $X$, and let $P\in X$ be an arbitrary point. Then the $\Q$-divisor 
\[
D \deq \frac{L_1+\ldots+L_m}{m}
\]
and the set $S=\{P\}$  satisfy the conditions of \cite[Theorem 5.1]{ghidelli_logarithmic_2021}. Consequently, for any given $0<\eps\ll 1$, there exists a positive rational number $t$ and  an effective $\Q$-divisor $\Delta\in |tD|_\Q$ such that\footnote{We refer to the notation in \cite{ghidelli_logarithmic_2021}. The relevant notation $F(n,r)$ is defined right before Theorem 4.1 in \emph{loc.~cit.}}
\begin{enumerate}
	\item $t<F(n,1)+\epsilon$,
	\item $(X,\Delta)$ is log canonical but not klt at $P$,
	\item $\text{LLC}(X,\Delta;P)=\{P\}$,
	\item $P\in\text{Nklt}(X,\Delta)$.
\end{enumerate}

The base-point freeness argument in the proof of \cite[Theorem 5.2]{ghidelli_logarithmic_2021} goes through verbatim for $\Delta$ as above, hence, it implies that $K_X+L_1+\ldots+L_m$ is globally generated for $m > F(n,1)$. This shows $\conFN(X) \leq F(n,1) + 1$. Upon combining this statement  with \cite[Theorem 4.9]{ghidelli_logarithmic_2021}, we obtain that 
\[
\conFN(X) \dleq F(n,1) + 1  \dleq 1+ n(\log(\log n)+2.34)\ ,
\]
as promised. The more precise bound for $n \leq 4$ follows from $\conFN(X) \leq \lfloor F(n,1) \rfloor + 1$, and the data in the table on page 12 of  \cite{ghidelli_logarithmic_2021}.
\end{proof}

\subsection{Fujita numbers of products}
Let $X$ and $Y$ be smooth projective varieties. For ample line bundles $\cL$ on $X$ and $\cM$ on $Y$ the line bundle $\cL \boxtimes \cM$ on $X \times Y$ is ample. Let $\cL_i$ (resp.\ $\cM_i$) be ample line bundles on $X$ (resp.\ on  $Y$) for $i = 1, \ldots, m$.  If the line bundle
\begin{equation}
\label{eq:kuenneth}
\omega_{X \times Y} \otimes \bigotimes_{i=1}^m (\cL_i \boxtimes \cM_i) = (\omega_X \otimes \bigotimes_{i=1}^m \cL_i ) \boxtimes (\omega_Y \otimes \bigotimes_{i=1}^m \cM_i) 
\end{equation}
is globally generated, then also its restriction $\omega_X \otimes \bigotimes_{i=1}^m \cL_i$  (resp.\ $\omega_Y \otimes \bigotimes_{i=1}^m \cM_i$) to the fiber of the projection $X \times Y \to X$ (resp.\ $X \times Y \to Y$) is globally generated. This immediately shows the following lemma.

\begin{lem} 
\label{lem:FNforProduct}
Let $X$ and $Y$ be smooth projective varieties. Then 
\[
\conFN(X \times Y) \geq \max\{\conFN(X), \conFN(Y) \}. 
\]
\end{lem}

The estimate can be improved to an equality in some favorable cases. Recall that the isogeny category of abelian varieties is a semisimple abelian category (by Poincar\'e's complete reducibility theorem). By an  isogeny factor of an abelian variety $A$ we mean a direct summand of $A$ in the isogeny category.

\begin{prop}
\label{prop:FNforProduct}
Let $X$ and $Y$ be smooth projective varieties such that the abelian varieties $\Pic^0_X$ and $\Pic^0_Y$ have no common nontrivial isogeny factor. Then  the following holds.
\begin{enumerate}[align=left,labelindent=0pt,leftmargin=*]
\item 
\label{propitem:FNforProduct1}
$\Pic(X \times Y) = \Pic(X) \times \Pic(Y)$, and 
\item
\label{propitem:FNforProduct2}
$\conFN(X \times Y) = \max\{\conFN(X), \conFN(Y) \}.$
\end{enumerate}
\end{prop}
\begin{proof}
We first show that \eqref{propitem:FNforProduct1} implies \eqref{propitem:FNforProduct2}. The product decomposition for $\Pic(-)$ is natural by restriction to fibers, hence also the nef cones and its interior, the ample cone, are products of the respective cones of the factors. Now arguing as in the proof of \lref{lem:FNforProduct}, it remains to see that $\omega_{X \times Y} \otimes \bigotimes_{i=1}^m (\cL_i \boxtimes \cM_i)$ is globally generated if the restrictions to the fibers of both projections are globally generated. This follows at once by the K\"unneth formula applied to \eqref{eq:kuenneth}.

For \eqref{propitem:FNforProduct1} we analyze the Leray spectral sequence for $\Gm$ along the projection $\pr: X \times Y \to X$. The low degree terms yield an exact sequence
\[
0 \to \Pic(X) \xrightarrow{\pr^\ast} \Pic(X \times Y) \to \rH^0(X,\R^1\pr_\ast \Gm) \xrightarrow{d_2^{0,1}} \rH^2(X,\Gm)  \xrightarrow{\pr^\ast} \rH^2(X \times Y, \Gm).
\]
Since the choice of a point $y \in Y$ and the map $i(x) = (x, y)$ splits the projection, the map $\pr^*$ admits a retraction. Hence the boundary map $d_2^{0,1}$ is the zero map and we have a short exact sequence
\[
0 \to \Pic(X) \xrightarrow{\pr^\ast} \Pic(X \times Y) \to \rH^0(X,\R^1\pr_\ast \Gm)  \to 0.
\]
The sheaf $\R^1\pr_\ast \Gm$ is represented by the Picard variety of $Y$, hence 
\[
\rH^0(X,\R^1\pr_\ast \Gm) = \Hom(X,\Pic_Y).
\]
Now we fix a point $x \in X$. Evaluation in $x$ and constant maps, as well as the Albanese property describe a canonical splitting
\[
\Hom(X,\Pic_Y) = \Pic(Y) \times \Hom((X, x), (\Pic_X,0)) = \Pic(Y) \times \Hom(\Alb_X, \Pic^0_Y).
\]
Because Albanese and $\Pic^0$ are dual abelian varieties and thus share the same isogeny factors, our assumption is precisely that the second factor vanishes. This proves \eqref{propitem:FNforProduct1}.
\end{proof}

We now prove \pref{propABC:FNproductwithPn} from the introduction.

\begin{cor}
\label{cor:FNproductwithPn}
Let $X$ be a smooth projective variety, and let $Y$ be a smooth projective variety with $\rH^{1}(Y,\dO_{Y}) = 0$.  
Then 
\[
\conFN(X \times Y) = \max\{\conFN(X), \conFN(Y)\}.
\]
\end{cor}
\begin{proof}
Since $\rH^{1}(Y,\dO_{Y}) = 0$, we have that $\Pic_Y^0$ is trivial, so \pref{prop:FNforProduct} proves the claim.
\end{proof}

\section{Simply connected varieties}
\label{sec:FujitaSimplyConnected}

\subsection{Complete intersections} 
In this section, let $X$ be an $n$-dimensional smooth complete intersection in $\bP^{n+r}$ of hypersurfaces  of degree $d_i$ for $i=1, \ldots, r$.

\begin{prop}
\label{prop:FNcompleteintersection}
If $n \geq 3$, then the complete intersection $X \inj \bP^{n+r}$ of multidegree $(d_1, \ldots, d_r)$ has convex Fujita number
\[
\conFN(X) = \max\big\{0,  (n+r+1) - \sum_{i=1}^r  d_i \big\}.
\]
\end{prop}
\begin{proof}
By the Lefschetz hyperplane theorem for $\Pic(-)$, 
see \cite[Exp.~XII Cor.~3.6]{grothendieck_cohomologie_1965}, 
the group $\Pic(X)$ is generated by $\dO(1)|_X$. The line bundle $\cL = \dO(a)|_X$ is ample if and only if $a \geq 1$, and it is globally generated if and only if $a \geq 0$. 

The adjoint bundle $\omega_X \otimes \cL$ for a product  $\cL = \cL_1 \otimes \ldots \otimes \cL_m$ of ample line bundles $\cL_i = \dO(a_i)|_X$ is by adjunction
\[
\omega_X \otimes \cL = \dO\big(- (n+r+1) + \sum_{i=1}^m a_i + \sum_{i=1}^r d_i \big)
\]
For a fixed $m$ all possible $a_i \geq 1$ lead to globally generated adjoint bundles if and only if the special case $a_i = 1$ for all $i=1, \ldots, m$ leads to a globally generated line bundle, equivalently if
\[
m  \geq (n+r+1) -  \sum_{i=1}^r d_i.
\]
This translates into the claimed formula for $\conFN(X)$.
\end{proof}

\subsection{Simply connected surfaces} While complete intersections can treat dimension $\geq 3$, for surfaces we argue with explicit examples.

\begin{prop}
\label{prop:FNsurfaces simply connected}
Let $0 \leq k \leq 3$ be an integer. Then there is a simply connected smooth projective surface $X$  with convex Fujita number $k$. 
\begin{enumerate}[align=left,labelindent=0pt,leftmargin=*]
\item 
\label{propitem:FNsurface sc1}
More concretely, for $1 \leq k \leq 3$ the blow up of $\bP^2$  in $3-k$ points has $\conFN(X) = k$.
\item 
\label{propitem:FNsurface sc2}
A very general hypersurface $X$ in $\bP^3$ of degree $d \geq 5$ is simply connected, has convex Fujita number $\conFN(X) = 0$ and is of general type.
\end{enumerate}
\end{prop}
\begin{proof}
\eqref{propitem:FNsurface sc1} These blow up surfaces are simply connected by birational invariance of  $\pi_1$. 
For $k=3$ we  deal with $X = \bP^2$ that has $\conFN(\bP^2) = 3$ because of the ample line bundle $\dO(1)$. 

\smallskip

For $k = 2$ we consider the blow up $X \to \bP^2$ in one point. This is the Hirzebruch surface $\bP(\dO \oplus \dO(-1))$, and $\Pic(X)$ is generated by the class of a fiber $F$ and the class of a section $S$ with self intersection $(S^2) = -1$. A divisor $L = aS + bF$ is nef if and only if $b \geq a \geq 0$. Consequently, a divisor $L = aS+bF$ is ample, i.e.\ in the interior of the nef cone, if and only if $b > a > 0$.  Moreover, as Hirzebruch surfaces are toric varieties, $aS + bF$ is globally generated if and only it is nef, see 
\cite[Theorem 3.1]{mustata_vanishing_2002}. The canonical class is $K = -2S - 3F$. For ample divisors $L_i = a_i S + b_i F$, $i=1,\ldots, m$, we find that
\[
K + \sum_{i=1}^m L_i = (-2 + \sum_{i=1}^m a_i )S + (-3 + \sum_{i=1}^m b_i ) F.
\] 
For $L_i = S + 2F$ this becomes $(m-2)S + (2m-3)F$, and that is nef, hence globally generated, if and only if  $m \geq 2$. When $m \geq 2$, for general ample divisor $L_i$, then 
\[
-3 + \sum_{i=1}^m b_i \geq -3 + \sum_{i=1}^m (a_i+1) > -2  + \sum_{i=1}^m a_i  \geq 0,
\]
and so the corresponding adjoint divisor is globally generated. This shows that $\conFN(X) = 2$.

\smallskip

For $k = 1$ we consider the blow up $X \to \bP^2$ in two points. This is a del Pezzo surface of degree $7$,
and $\Pic(X)$ is generated by the pullback $H$ of the line and the two exceptional fibers $E_1$ and $E_2$. A divisor $L = dH - a_1 E_1 -a_2E_2$ is nef,  if and only if 
\[
d \geq a_1 + a_2 \quad  \text{ and}  \quad a_i \geq 0, \text{ for $i = 1,2$}.
\]
Consequently, $L$ being ample means that all inequalities are strict. Since we may think of $X$ as being the blow up in two torus invariant points, $X$ is a toric variety and so again by \cite[Theorem 3.1]{mustata_vanishing_2002} a divisor $L$ is globally generated if and only if $L$ is nef. The canonical divisor is $K = -3H + E_1 + E_2$ and not globally generated, hence $\conFN(X) \geq 1$. For any ample divisor $L = dH - a_1E_1 - a_2 E_2$, i.e.\ $d \geq a_1 + a_2 + 1$ and $a_i \geq 1$, the adjoint divisor
\[
K + L = (d-3)H - (a_1-1)E_1 - (a_2-1)E_2 
\]
has $(a_i-1) \geq 0$ and 
\[
d-3 \geq a_1 + a_2 + 1 - 3  = (a_1 - 1) + (a_2 - 1),
\]
hence $K+L$ is globally generated. This shows that $\conFN(X) = 1$.

\smallskip

\eqref{propitem:FNsurface sc2}  The hyperplane is simply connected due to Lefschetz hyperplane theorem for the fundamental group. By Noether-Lefschetz \cite{lefschetz_certain_1921} and degree $d \geq 4$ (see also \cite[Exp.~XIX Th\'eor\`eme 1.2]{deligne_groupes_1973}), the Picard group $\Pic(X)$ is generated by $\dO(1)|_X$. This is the reason to restrict to very general hyperplanes. The convex Fujita number is then calculated as in the proof of \pref{prop:FNcompleteintersection}. Since the degree is at least $d \geq 5$, the canonical bundle $\omega_X = \dO(d-4)|_X$ is very ample.
\end{proof}

\subsection{Convex Fujita numbers for simply connected varieties}

We now prove \tref{thmABC:FNsimplyconnected}.

\begin{thm}
\label{thm:FNsimplyconnected}
Let  $n \geq 2$ be an integer, and let $0 \leq k \leq n+1$. There is a simply connected smooth projective variety $X$ of dimension $n$ with convex Fujita number $\conFN(X) = k$. Moreover, if $k=0$, then $X$ can be chosen to be of general type.
\end{thm}
\begin{proof}
We remark first that $n$-dimensional complete intersections in $\bP^{n+r}$ are simply connected by the Lefschetz hyperplane theorem as long as $n \geq 2$. For $n \geq 3$ thus the proof reduces to \pref{prop:FNcompleteintersection}, because all values of $k$ can be obtained, e.g. by $r = 1$ and $d_1 = n+2-k$. If $k=0$, then we choose $r=1$ and $d_1 > n+2$ which forces $\omega_X$ to be very ample and $X$ to be of general type. 

The case $n=2$ is nothing but \pref{prop:FNsurfaces simply connected}.
\end{proof}

\section{Varieties with prescribed fundamental group}
\label{sec:FujitaSimple}

In this section we construct smooth projective varieties with a given fundamental group and varying dimension and convex Fujita number.

\subsection{Fujita simple surfaces with prescribed fundamental group}

We now prove the surface case of \tref{thmABC:FNsimple}.

\begin{prop}
\label{prop:FNsimpleSurface}
Let $\pi$ be a projective group. Then there is a smooth projective surface $X$ of general type with fundamental group isomorphic to $\pi$ and convex Fujita number $\conFN(X)= 0$.
\end{prop}
\begin{proof}
Let $Y$ be a smooth projective variety with $\pi_1(Y) = \pi$. Upon replacing $Y$ by its product with some projective space, we may assume $\dim(Y) \geq 3$.  Passing to a smooth complete intersection of $\dim(Y)-3$ very ample divisors if necessary, we assume going forward that $\dim(Y) = 3$. 

Fixing an ample divisor $H$ on $Y$ and an integer $p \geq 5$ such that $K_Y + pH$ is ample and globally generated, we choose $X$ to be a very general smooth hypersurface in the linear system $|pH|$.  By the Lefschetz hyperplane theorem, $\pi_1(X)$ is isomorphic to $\pi$. By Noether-Lefschetz for arbitrary $3$-folds, see \cite{joshi_noether-lefschetz_1995} and more effectively using global generation of $K_Y + pH$ by \cite[Theorem~1]{ravindra_girivaru_v_noetherlefschetz_2009},
we have 
\[
\Pic(X) = \Pic(Y)
\]
via restriction. This means that for all $\cL_1, \cL_2 \in \Pic(X)$ there are extensions $\cM_1, \cM_2 \in \Pic(Y)$ with $\cL_i = \cM_i|_X$ and thus
\[
(\cL_1 \cdot  \cL_2)_X = (\cM_1 \cdot  \cM_2 \cdot \dO_Y(pH))_Y \in p\bZ
\]
is divisible by $p$.  This shows $\conFN(X) \leq 1$ by \pref{prop:FNsurfacesReider divisible intersection numbers}.

By adjunction, the canonical bundle of $X$ is $\omega_X = (\omega_Y \otimes \dO_Y(pH))|_X$. Since $\omega_Y \otimes \dO_Y(pH)$ is ample and globally generated, its restriction to $X$ is also ample and globally generated, so $\conFN(X) = 0$ as desired.
\end{proof}

 \tref{thmABC:FNsimple} now follows easily with the exception of dimension $n=3$. Indeed, given a projective group $\pi$ we first choose a Fujita simple surface $S$ of general type with $\pi_1(S) = \pi$ as in \pref{prop:FNsimpleSurface}. If $n=2$ we are done. Otherwise by \tref{thmABC:FNsimplyconnected}, since we excluded $n=3$, we have a complete intersection $Y$ of general type and dimension $n-2$ which is Fujita simple and simply connected. By  \pref{propABC:FNproductwithPn} the variety $X = S \times Y$ is Fujita simple, it is of general type and dimension $n$, and with $\pi_1(X) = \pi$ by the K\"unneth formula for $\pi_1$. This constructs the variety asked for in  \tref{thmABC:FNsimple}. 

\subsection{Totally branched cyclic covers}

It remains to construct a threefold that satisfies the needs of \tref{thmABC:FNsimple}. Since the construction is not particularly different in all dimensions $n \geq 3$, we do not specialize to  
threefolds now. We dealt with surfaces first in \pref{prop:FNsimpleSurface} because we wanted to highlight how far one gets by only invoking Reider's method instead of the analytic methods of Angehrn-Siu \cite{angehrn_effective_1995}.

\begin{rmk}
We first recall the construction of the cyclic and totally branched covering $X \to Y$ of a smooth projective variety $Y$ with respect to a line bundle $\cL \in \Pic(Y)$, a degree $d \geq 1$, and a smooth divisor $B$ (branch locus) in the linear system associated to $\cL^{\otimes d}$, more precisely to an isomorphism 
\[
s : \cL^{\otimes d} \xrightarrow{\sim} \dO_Y(B).
\]
The $\mu_d$-torsor $X \to Y$ is constructed as the relative spectrum
\[
f: X = \Spec_Y\big(\Sym^\bullet(\cL^{-1})/\cL^{\otimes -d} = \cO_Y(-B)\big) \to Y.
\]
which, locally with an equation $\{s=0\}$ for the Cartier divisor $B$, solves the equation $t^d = s$ inside the line bundle $\cL$ globally. The variety $X$ is again smooth projective by Abhyankar's lemma \cite[Exp.~XIII Proposition~5.1]{SGA1}, which also shows that the ramification locus $R = f^{-1}(B)_\redu$  is smooth, in fact isomorphic to $B$. We have the well known relations
\begin{align*}
	\cO_X(R) & \simeq f^\ast \cL, \\ 
	f^\ast B & =  d \cdot R, \\
	\omega_X(R) & \simeq f^\ast (\omega_Y(B)).
\end{align*}
It follows that  
\begin{equation}
	\label{eq:ramifiedcanonical bundle formula}
	\omega_X \simeq f^\ast(\omega_Y(B)) \otimes \dO_X(-R)  \simeq f^\ast \big(\omega_Y \otimes \cL^{\otimes (d-1)}\big).
\end{equation}
\end{rmk}

The next lemmas record the effect of various Lefschetz theorems applied to functors evaluated at the following diagram showing the cyclic totally branched covering constructed above. 
\begin{equation}
	\label{eq:branchedcovering}
	\xymatrix@M+1ex{
		R \ar@{^(->}[r] \ar[d]^{f|_R} \ar[d]_{\simeq} & X \ar[d]^f \\
		B \ar@{^(->}[r] & Y.
	}
\end{equation}

\begin{lem}
	\label{lem:pi1 in branched cover}
	Let $f: X \to Y$ be the branched $\mu_d$-cover constructed above associated to $\cL$ and $B$. If $\cL$ is ample and the dimension $\dim(Y)$ is at least $3$, then $f$ induces an isomorphism
	\[
	f_\ast : \pi_1(X) \xrightarrow{\sim} \pi_1(Y).
	\]
\end{lem}
\begin{proof}
	Since $\cL$ is ample, both $B$ in $Y$ and $R$ in $X$ are ample. By the Lefschetz hyperplane theorem, the functor $\pi_1$ applied to diagram \eqref{eq:branchedcovering} has isomorphisms for all but the map induced by  $f$. Hence also $f_\ast$ is an isomorphism.
\end{proof}

\begin{rmk}
\lref{lem:pi1 in branched cover} also holds in case $\dim(Y) = 2$ and then with the weaker assumption that $B$ is a smooth divisor with $B^2 > 0$ instead of being an ample divisor. This is proven by Kharlamov and Kulikov in \cite[Proposition 1]{kulikov_numerically_2014} based on work of Nori. We will quickly sketch the argument. By an application of Nori's weak Lefschetz theorem, namely  \cite[Proposition 3.26]{nori_zariski_1983}, the natural map $\pi_1(B) \to \pi_1(Y)$ has image of index bounded by $B^2/(B^2 - 2r(B))$, where $r(B) = 0$ since $B$ is smooth. Hence the map is still surjective and consequently $f_\ast: \pi_1(X) \to \pi_1(Y)$ is surjective, since $B$ lifts to $Y$ as $B \simeq R \to Y$. In order to show that $f_\ast$ is injective we consider the complements $U = X  - R$ and $V = Y - B$. The restriction $f|_U : U \to V$ is finite \'etale, hence $f|_{U,\ast} : \pi_1(U) \to \pi_1(V)$ is injective, in fact with cokernel isomorphic to $\mu_d = \Aut(U/V)$. We obtain a map of short exact sequences with $N_R$ and $N_B$ implicitly defined as the respective kernels
\[
\xymatrix@M+1ex@R-2ex{
1 \ar[r] & N_R \ar[d] \ar[r] & \pi_1(U) \ar[d]^{f|_{U,\ast} } \ar[r] & \pi_1(X) \ar[d]^{f_\ast} \ar[r] & 1 \\
1 \ar[r] & N_B \ar[r] & \pi_1(V) \ar[r] & \pi_1(Y) \ar[r] & 1  \ .
}
\]
The rows of the diagram are central extensions by \cite[Proposition 6.5]{nori_zariski_1983}. 
Kharlamov and Kulikov analyze the proof of  \cite[Proposition 3.27]{nori_zariski_1983} and show that $N_B/N_R$ is isomorphic to $\mu_d$ because $N_B$ and $N_R$ are abelian and generated by the respective inertia generators of a small loop around $B$ (resp.\ $R$). The snake lemma (valid also for non-abelian groups) completes the proof that $f_\ast$ is an isomorphism.
\end{rmk}

\begin{lem}
	\label{lem:Pic in branched cover}
	Let $f: X \to Y$ be the branched $\mu_d$-cover constructed above associated to $\cL$ and $B$. If $\cL$ is ample and the dimension $\dim(Y)$ is at least $4$, then $f$ induces an isomorphism
	\[
	f^\ast : \Pic(Y) \xrightarrow{\sim} \Pic(X).
	\]
Moreover, if $d \gg 0$ is sufficiently large and in addition $B$ is chosen to be very general, then the same conclusion holds when $\dim(Y) \geq 3$.
\end{lem}
\begin{proof}
Again, since $\cL$ is ample, both $B$ in $Y$ and $R$ in $X$ are ample. 

Let first $\dim(Y)$ be at least $4$. By the Lefschetz hyperplane theorem for $\Pic(-)$, 
see \cite[Exp.~XII Cor.~3.6]{grothendieck_cohomologie_1965}, the functor $\Pic$ applied to diagram \eqref{eq:branchedcovering} has isomorphisms for all but the map induced by  $f$. Hence also $f^\ast$ is an isomorphism.

For $\dim(Y) = 3$ and very general $B$ we use again the Noether-Lefschetz theorem of \cite{joshi_noether-lefschetz_1995} to deduce that the restriction $\Pic(Y) \to \Pic(B)$ is an isomorphism. We cannot use the same argument for the restriction $\Pic(X) \to \Pic(R)$, since we do not control whether $R$ is very general in $X$, and most likely it is not. But by the Lefschetz hyperplane theorem for $\Pic(-)$, 
see \cite[Exp.~XII Cor.~3.6]{grothendieck_cohomologie_1965}, the restriction $\Pic(X) \inj \Pic(R)$ is still injective, and that is sufficient to conclude as in the proof before when $\dim(Y) \geq 4$.
\end{proof}

\begin{prop}
	\label{prop:FNin branched cover}
	Let $f: X \to Y$ be the branched $\mu_d$-cover constructed above associated to an ample line bundle $\cL$ and a smooth divisor $B$. We assume that the dimension $\dim(Y)$ is at least $3$, where if $\dim(Y) = 3$ we ask $d \gg 0$ to be large and $B$ to be very general. Then the convex Fujita number is bounded as follows:
	\[
	\conFN(X) \leq \max\{0, \conFN(Y) + 1 - d \}.
	\]
Moreover,  if $d - 2 \geq \conFN(Y)$, then $\omega_X$ is ample and globally generated. 
\end{prop}
\begin{proof}
	Let $\cL_1, \ldots, \cL_s$ be ample line bundles on $Y$. Due to Lemma~\ref{lem:Pic in branched cover} there are line bundles $\cM_i$ on $X$ with $\cL_i = f^\ast \cM_i$. Moreover, as $f$ is finite, the line bundles $\cM_i$ are also  ample. Since
	\[
	\omega_X \otimes (\cL_1 \otimes \ldots \otimes \cL_s) \simeq f^\ast \big(\omega_Y \otimes \cL^{\otimes (d-1)} \otimes  \cM_1 \otimes \ldots \otimes \cM_s \big),
	\]
	and because being globally generated pulls back under morphisms, the left hand side is globally generated as soon as $s + d - 1 \geq \conFN(Y)$. The estimate for $\conFN(X)$ follows.
	
If $d - 2 \geq \conFN(X)$, then $\omega_Y \otimes \cL^{\otimes (d-2)}$ is globally generated. The tensor product of an ample line bundle $\cL$ with a globally generated line bundle $\omega_Y \otimes \cL^{\otimes (d-2)}$ is again ample, hence $\omega_X$ is ample as the pull back of  $\omega_Y \otimes \cL^{\otimes (d-1)}$.
\end{proof}

\subsection{Fujita simple varieties with prescribed fundamental group and dimension}

The proof of \tref{thmABC:FNsimple} begun in \pref{prop:FNsimpleSurface} will now be completed by the following proposition.

\begin{prop}
\label{prop:FNsimpleVariety}
Let $\pi$ be a projective group and $n \geq 3$. Then there is a smooth projective variety $X$ of general type and dimension $n$ with fundamental group isomorphic to $\pi$ and convex Fujita number $\conFN(X)= 0$.
\end{prop}
\begin{proof}
We can argue as in the proof of the surface case \pref{prop:FNsimpleSurface} that there is a smooth projective variety $Y$ of general type with $\pi_1(Y) = \pi$ and dimension $n$. Now we choose an ample line bundle $\cL$ on $Y$.

By \pref{prop:FNisFinite}, the convex Fujita number $\conFN(Y)$ is finite. We choose $d \geq \conFN(Y) + 2$ large enough such that by Bertini we find a smooth divisor $B$ in the linear system associated to $\cL^{\otimes d}$. Let $f: X \to Y$ be the branched $\mu_d$-cover constructed above associated to $\cL$ and $B$.  If $\dim(Y) = 3$ we moreover ask $d \gg 0$ to be sufficiently large and $B$ very general, so that the conclusion of 
\lref{lem:Pic in branched cover} holds. Then \pref{prop:FNin branched cover} shows that $X$ is Fujita simple and of general type.
\end{proof}

In a certain range we may improve on \tref{thmABC:FNsimple}  by even imposing the value of the Kodaira dimension. 

\begin{thm}
\label{thm:FNsimpleVarietyKodaira m}
Let $\pi$ be a projective group, let $n \geq 4$ and let $n-2 \geq m \geq 2$. Then there is a smooth projective variety $X$ of Kodaira dimension $m$ and dimension $n$ with fundamental group isomorphic to $\pi$ and convex Fujita number $\conFN(X)= 0$.
\end{thm}
\begin{proof}
We construct $X$ as a product $X = Y \times Z$. \tref{thmABC:FNsimple} provides a smooth projective $Y$ of general type and  dimension $m$ with $\pi_1(Y)$ isomorphic to $\pi$ and $\conFN(Y) = 0$. 

The factor $Z$ is obtained as a smooth hypersurface $Z \inj \bP^{n+1-m}$ of degree $n+2-m$ that, moreover, we require to be very general if $n-m = 2$. If $\dim(Z) = 2$, then $Z$ is a very general smooth quartic in $\bP^3$, hence a K3 surface with Picard group generated by $\dO(1)|_Z$. The argument of \pref{prop:FNcompleteintersection} applies to show $\conFN(Z) = 0$. If $\dim(Z) \geq 3$, then \pref{prop:FNcompleteintersection} applies directly to show $\conFN(Z) = 0$, too. 

By the Lefschetz hyperplane theorem for the fundamental group we have $\pi_1(Z) = 0$ so that $\pi_1(X)$ is isomorphic to $\pi_1(Y) \simeq \pi$. Being simply connected, $Z$ also has vanishing $\rH^1(Z,\dO_Z)$ and so \cref{cor:FNproductwithPn} yields $\conFN(X) = \max\{\conFN(Y),\conFN(Z)\} = 0$. 

As $Z$ has trivial canonical bundle, we have $\omega_X = \pr^\ast \omega_Y$ with $\pr : X = Y \times Z \to Y$ the projection map. Therefore $X$ and $Y$ have the same Kodaira dimension, namely $m$ by construction. 
\end{proof}

\subsection{Convex Fujita numbers with prescribed fundamental group and dimension}
We now aim to prove \tref{thmABC:anyFNanyPi1} of the introduction.

\begin{thm}
\label{thm:anyFNanyPi1}
Let $\pi$ be a projective group, and let $n \geq 2$ be an integer, and let $0 \leq k \leq n-1$. There is a smooth connected projective variety $X$ of dimension $n$ with the following properties:
\begin{enumerate}[label=(\roman*),align=left,labelindent=0pt,leftmargin=*,widest = (iii)]
\item
$\pi_1(X)$ is isomorphic to $\pi$, and 
\item
$\conFN(X) = k$.
\end{enumerate}
\end{thm}
\begin{proof}
The case $k=0$ is the content of  \tref{thmABC:FNsimple}, so that case is done. Let $S$ be such a Fujita simple surface with $\pi_1(S) = \pi$. 

Let first $n$ be at least $4$. If $0 \leq k \leq n-1$, then there exists a simply connected variety $Y$ of dimension $n-2 \geq 2$ and convex Fujita number $k$ by \tref{thmABC:FNsimplyconnected}. By the K\"unneth formula the product $X = S \times Y$ has $\pi_1(X) = \pi_1(S) = \pi$, and, since $\rH^{1}(Y,\dO_{Y})=0$ for the simply connected variety $Y$, we have from \cref{cor:FNproductwithPn} that $\conFN(X) = \max\{0,k\} = k$.

If $n = 3$ and $k=2$, then we can take $Y = \bP^1$ with $\conFN(Y) = 2$ and again conclude that $X = S \times Y$ has the required fundamental group and convex Fujita number $2$. 

The only cases missing now are the case $(n,k) = (2,1)$ and $(3,1)$ which we deal with separately in \pref{prop:FNn2k1} and \pref{prop:FNn3k1} below.
\end{proof}

\begin{prop}
\label{prop:FNn2k1}
Let $\pi$ be a projective group. There exists a smooth projective surface $S$ such that the blow-up $S' \to S$ in a point yields a smooth projective surface with $\conFN(S') = 1$ and $\pi_1(S') =\pi$.
\end{prop}
\begin{proof}
We can argue as in the proof of  \pref{prop:FNsimpleSurface} that there is a smooth projective variety $Y$ of general type with $\pi_1(Y) = \pi$ and dimension $3$. Now we choose a very ample line bundle $\dO(1)$ on $Y$ and choose $S$ a very general smooth divisor in the linear system of $\dO(24)$ by Bertini's theorem. Upon replacing $\dO(1)$ by a multiple initially, we may assume by Noether-Lefschetz, see \cite{joshi_noether-lefschetz_1995}, that $\Pic(S) = \Pic(Y)$. The choice of the power $24$ then forces all intersection numbers of line bundles on $S$ to be divisible by $24$.

By the Lefschetz hyperplane theorem and the birational invariance of the fundamental group the blow up $\sigma: S' \to S$ in a choice of a point $P$ on $S$ is a smooth projective surface $S'$ with 
\[
\pi_1(S') = \pi_1(S) = \pi_1(Y) = \pi.
\]

Since $S'$ is not minimal the canonical divisor $K_{S'}$ is not nef and a fortiori not globally generated. Thus $\conFN(S') \geq 1$. It remains to show that for all ample divisors $L$ on $S'$ the adjoint divisor $K_{S'} + L$ is globally generated. 

Since $\Pic(S')$ equals $\Pic(S) \oplus \bZ$, with the summand $\bZ$ spanned by the exceptional divisor $E = \sigma^{-1}(P)$, we find that $L = \sigma^\ast M -aE$ for some divisor $M$ on $S$  and $a \in \bZ$. We compute
\[
(L^2) = (M^2) - a^2 > 0,
\]
so modulo $24$ it's the negative of a square. Squares modulo $24$ are $0,1,4,9,12,16$, hence 
\[
(L^2) \geq 8.
\]
Reider's theorem \cite[Theorem 1 (i)]{reider_vector_1988} implies that $K_{S'}  + L$ is globally generated unless we are in the exceptional case: there is an effective divisor $C'$ on $S'$ such that $(C'^2) = 0$ and $(L \cdot C') = 1$. Since $L$ is ample and $C'$ is effective, it follows from $(L \cdot  C') = 1$  that $C'$ must be irreducible and reduced. As $(E^2) = -1$ we can exclude $C' = E$. Hence $C'$ is the strict transform of an irreducible and reduced curve $C$ on $S$. Let $m_C$ be the multiplicity of $C$ in $P$. Then 
\[
0 = (C'^2) = (C^2) - m_C^2 \equiv -m_C^2  \pmod{24}.
\]
It follows that $12$ divides $m_C$. On the other hand, we have $(M \cdot C) \equiv 0 \pmod{24}$ and so
\[
1 = (L \cdot C') =  (M \cdot C) - a m_C \equiv 0 \pmod{12},
\]
a contradiction. 
\end{proof}

\begin{prop}
\label{prop:FNn3k1}
Let $\pi$ be a projective group, and let $S$ be a smooth projective surface with $\conFN(S) \leq 2$. Let $\dM$ be a very ample line bundle on $S$, and set $\cL = \dM \boxtimes \dO(1)$ on $Y = S \times \bP^1$. Then the double cover $f:X \to Y$ branched in a very general smooth divisor $B$ of the linear system associated to $\cL^{\otimes 2}$ is a smooth projective threefold with $\conFN(X) = 1$ and $\pi_1(X) =\pi$.
\end{prop}
\begin{proof}
Let $\pr: Y \to S$ be the projection map. Since $\conFN(S) \leq 2$, the line bundle
\[
\omega_Y(B) = \big(\omega_S \otimes \cM^{\otimes 2}\big) \boxtimes \big(\omega_{\bP^1} (2)\big) = \pr^\ast(\omega_S \otimes \cM^{\otimes 2})
\]
is globally generated. Therefore the  effective version of the Noether-Lefschetz theorem proved in 
\cite[Theorem~1]{ravindra_girivaru_v_noetherlefschetz_2009} applied to $\dO_Y(1) \coloneq \cL^{\otimes 2}$ shows that the restriction map 
\[
\Pic(Y) \xrightarrow{\sim} \Pic(B)
\]
is an isomorphism. This suffices to show as in \lref{lem:Pic in branched cover} that
\[
f^\ast : \Pic(Y) \to \Pic(X)
\]
is an isomorphism. By \pref{propABC:FNproductwithPn} we have $\conFN(Y) = 2$, and by 
 \pref{prop:FNin branched cover} we have $\conFN(X) \leq 1$. It remains to show that $\omega_X$ is not globally generated. We compute its global sections using $f_\ast \dO_X = \dO_Y \oplus \cL^{-1}$ as
 \[
 \rH^0(X,\omega_X) = \rH^0(X, f^\ast(\omega_Y \otimes \cL)) = \rH^0(Y, \omega_Y \otimes \cL) \oplus \rH^0(Y,\omega_Y).
 \]
 Furthermore, by the K\"unneth formula and the product structure of $Y$ and $\omega_Y \otimes \cL$ 
 \begin{align*}
 \rH^0(Y, \omega_Y \otimes \cL) & = \rH^0(S,\omega_S \otimes \cM) \otimes \rH^0(\bP^1,\dO(-1)) = 0, \\
  \rH^0(Y,\omega_Y) & = \rH^0(S,\omega_S) \otimes \rH^0(\bP^1,\dO(-2)) = 0, 
 \end{align*}
 both vanish. Hence $\omega_X$ has no global sections and the proof is complete.  
\end{proof}

\begin{rmk}
A result similar in nature to \tref{thm:anyFNanyPi1} can be found in 
\cite[Proposition~26]{Debarre2005:VarietiesAmpleCotangent}. Based on earlier ideas and results due to Bogomolov, Debarre proves that for every projective group $\pi$ there is a smooth projective surface  with ample cotangent bundle and fundamental group isomorphic to $\pi$.
\end{rmk}

%

\begin{bibdiv}
\begin{biblist}

\bib{angehrn_effective_1995}{article}{
      author={Angehrn, Urban},
      author={Siu, Yum~Tong},
       title={Effective freeness and point separation for adjoint bundles},
      date={1995},
      journal={Inventiones Mathematicae},
      volume={122},
      number={2},
      pages={291\ndash 308},
}

\bib{bauer_tensor_1996}{article}{
      author={Bauer, Thomas},
      author={Szemberg, Tomasz},
       title={On tensor products of ample line bundles on abelian varieties},
       date={1996},
      journal={Mathematische Zeitschrift},
      volume={223},
      number={1},
      pages={79\ndash 85},
}

\bib{Debarre2005:VarietiesAmpleCotangent}{article}{
  author = {Debarre, Olivier},
  title = {Varieties with ample cotangent bundle},
  date = {2005},
  journal = {Compositio Mathematica},
  volume = {141},
  number = {6},
  pages = {1445\ndash 1459},
}

\bib{ein_global_1993}{article}{
      author={Ein, Lawrence},
      author={Lazarsfeld, Robert},
       title={Global generation of pluricanonical and adjoint linear series on smooth projective threefolds},
       date={1993},
     journal={Journal of the American Mathematical Society},
      volume={6},
      number={4},
      pages={875\ndash 903},
}

\bib{ein_syzygies_1993}{article}{
      author={Ein, Lawrence},
      author={Lazarsfeld, Robert},
      title={Syzygies and {Koszul} cohomology of smooth projective varieties of arbitrary dimension},
      date={1993},
     journal={Inventiones Mathematicae},
      volume={111},
      number={1},
      pages={51\ndash 67},
}

\bib{fujita_polarized_1987}{article}{
      author={Fujita, Takao},
       title={On {polarized} {manifolds} {whose} {adjoint} {bundles} {are} {not} {semipositive}},
       date={1987},
     journal={Algebraic Geometry, Sendai, 1985},
      volume={10},
       pages={167\ndash 179},
}

\bib{ghidelli_logarithmic_2021}{article}{
      author={Ghidelli, Luca},
      author={Lacini, Justin},
       title={Logarithmic bounds on {Fujita}'s conjecture},
       FJOURNAL={Proceedings of the London Mathematical Society. Third Series},
      date={2024},
      journal={Proc. Lond. Math. Soc. (3)},
      volume={128},
	NUMBER = {3},      
	pages={Paper No. e12591, 28},
}

\bib{heier_effective_2002}{article}{
      author={Heier, Gordon},
      title={Effective freeness of adjoint line bundles},
      date={2002},
     journal={Documenta Mathematica},
      volume={7},
      pages={31\ndash 42},
}

\bib{helmke_fujitas_1997}{article}{
      author={Helmke, Stefan},
       title={On {Fujita}'s conjecture},
        date={1997},
     journal={Duke Mathematical Journal},
      volume={88},
      number={2},
       pages={201\ndash 216},
}

\bib{joshi_noether-lefschetz_1995}{article}{
      author={Joshi, Kirti},
       title={A {Noether}-{Lefschetz} theorem and applications},
        date={1995},
     journal={Journal of Algebraic Geometry},
      volume={4},
      number={1},
       pages={105\ndash 135},
}

\bib{kawamata_fujitas_1997}{article}{
      author={Kawamata, Yujiro},
       title={On {Fujita}'s freeness conjecture for 3-folds and 4-folds},
        date={1997},
     journal={Mathematische Annalen},
      volume={308},
      number={3},
       pages={491\ndash 505},
}

\bib{kulikov_numerically_2014}{article}{
      author={Kulikov, Viktor~Stepanovich},
      author={Kharlamov, Viatcheslav},
       title={On numerically pluricanonical cyclic coverings},
        date={2014},
     journal={Izvestiya: Mathematics},
      volume={78},
      number={5},
       pages={986},
}

\bib{kuronya_effective_2022}{misc}{
      author={K{\"u}ronya, Alex},
      author={Mustopa, Yusuf},
       title={Effective global generation on varieties with numerically trivial canonical class},
       publisher={arXiv:1810.07079},
       year={2022},
}

\bib{laterveer_linear_1996}{article}{
      author={Laterveer, Robert},
       title={Linear systems on toric varieties},
        date={1996},
     journal={The Tohoku Mathematical Journal. Second Series},
      volume={48},
      number={3},
       pages={451\ndash 458},
}

\bib{lefschetz_certain_1921}{article}{
      author={Lefschetz, Solomon},
       title={On certain numerical invariants of algebraic varieties with
  application to abelian varieties},
        date={1921},
     journal={Transactions of the American Mathematical Society},
      volume={22},
      number={3},
       pages={327\ndash 406},
}

\bib{mustata_vanishing_2002}{article}{
      author={Musta\c{t}\u{a}, Mircea},
       title={Vanishing theorems on toric varieties},
        date={2002},
     journal={The Tohoku Mathematical Journal. Second Series},
      volume={54},
      number={3},
       pages={451\ndash 470},
}

\bib{nori_zariski_1983}{article}{
	author = {Nori, Madhav V.},
     	title = {Zariski's conjecture and related problems},
	date = {1983},
   	journal = {Ann. Sci. \'Ecole Norm. Sup. (4)},
	volume = {16}, 
	number = {2},
	pages = {305 \ndash 344}
}

\bib{pareschi_regularity_2003}{article}{
      author={Pareschi, Giuseppe},
      author={Popa, Mihnea},
       title={Regularity on abelian varieties. {I}},
        date={2003},
     journal={Journal of the American Mathematical Society},
      volume={16},
      number={2},
       pages={285\ndash 302},
}

\bib{reider_vector_1988}{article}{
      author={Reider, Igor},
       title={Vector {Bundles} of {Rank} 2 and {Linear} {Systems} on   {Algebraic} {Surfaces}},
        date={1988},
     journal={The Annals of Mathematics (2)},
      volume={127},
      number={2},
       pages={309\ndash 316},
}

\bib{ries_non-divisorial_2021}{article}{
      author={Rie{\ss}, Ulrike},
       title={On the non-divisorial base locus of big and nef line bundles on
  {K3}$^{\textrm{[2]}}$-type varieties},
        date={2021},
     journal={Proceedings of the Royal Society of Edinburgh. Section A.  Mathematics},
      volume={151},
      number={1},
       pages={52\ndash 78},
}

\bib{ravindra_girivaru_v_noetherlefschetz_2009}{article}{
      author={Ravindra, Girivaru~V.},
      author={Srinivas, Vasudevan},
       title={The {Noether}--{Lefschetz} theorem for the divisor class group},
        date={2009},
     journal={Journal of Algebra},
      volume={322},
      number={9},
       pages={3373\ndash 3391},
}

\bib{SGA1}{book}{
         editor={Grothendieck, Alexander},
         TITLE = {Rev\^{e}tements \'{e}tales et groupe fondamental},
	SERIES = {Lecture Notes in Mathematics, Vol. 224},
   	Author = {Alexandre Grothendieck},
     	NOTE = {S\'{e}minaire de G\'{e}om\'{e}trie Alg\'{e}brique du Bois Marie 1960--1961
              (SGA 1),
              Dirig\'{e} par Alexandre Grothendieck. Augment\'{e} de deux expos\'{e}s de
              Mich{\`{e}}le Raynaud},
 	PUBLISHER = {Springer-Verlag, Berlin-New York},
      	YEAR = {1971},
     	PAGES = {xxii+447},
        	label={SGA 1},
}


\bib{grothendieck_cohomologie_1965}{book}{
      editor={Grothendieck, Alexander},
       title={Cohomologie locale des faisceaux coh{\'e}rents et th{\'e}or{\`e}mes de
  {Lefschetz} locaux et globaux},
      publisher={Institut des Hautes {\'E}tudes Scientifiques, Paris},
      date={1965},
      label={SGA 2},
}


\bib{deligne_groupes_1973}{book}{
      author={Deligne, Pierre},
      author={Katz, Nicholas~M},
      title={Groupes de monodromie en g{\'e}om{\'e}trie alg{\'e}brique. {II}},
      series={Lecture {Notes} in {Mathematics}, {Vol}. 340},
      publisher={Springer-Verlag, Berlin-New York},
      label={SGA $7_{II}$},
      date={1973},
}


\bib{ye_fujitas_2020}{article}{
      author={Ye, Fei},
      author={Zhu, Zhixian},
       title={On {Fujita}'s freeness conjecture in dimension 5},
        date={2020},
     journal={Advances in Mathematics},
      volume={371},
       pages={107210, 56pp},
}

\end{biblist}
\end{bibdiv}


\end{document}